\documentclass[reqno,b5paper]{amsart}
\usepackage{amsmath}
\usepackage{amssymb}
\usepackage{amsthm}
\usepackage{enumerate}
\usepackage[mathscr]{eucal}
\newtheorem{theorem}{Theorem}[section]
\newtheorem{remark}{Remark}[section]
\newtheorem{example}{Example}[section]

\newtheorem{lemma}{Lemma}[section]
\newtheorem{proposition}{Proposition}[section]

\numberwithin{equation}{section}
\begin{document}

\title[Besov-Dunkl spaces connected with generalized Taylor formula]{Besov-Dunkl spaces connected with generalized Taylor formula on the real line}

\author[C. Abdelkefi \and F. Rached]{Chokri Abdelkefi* \and Faten Rached**}
\newcommand{\acr}{\newline\indent}
\address{\llap{*\,}
Department of Mathematics\acr Preparatory Institute of Engineer
Studies of Tunis \acr 1089 Monfleury Tunis, University of Tunis\acr
Tunisia} \email{chokri.abdelkefi@yahoo.fr}
\address{\llap{**\,}
Department of Mathematics\acr Preparatory Institute of Engineer
Studies of Tunis \acr 1089 Monfleury Tunis, University of Tunis\acr
Tunisia}
\email{rached@math.jussieu.fr}
\thanks{This work was completed with the support of the DGRST research project LR11ES11,
University of Tunis El Manar, Tunisia.}

\begin{abstract}
In the present paper, we define for the Dunkl tranlation operators
on the real line, the Besov-Dunkl space of functions for which the
remainder in the generalized Taylor's formula has a given order. We
provide characterization of these spaces by the Dunkl convolution.
\end{abstract}
\maketitle
\section{Introduction}
\label{intro} There are many ways to define the Besov spaces (see
\cite{An,Bes,Pe}) and the Besov-Dunkl spaces (see
\cite{ab1,ab2,ab3}). It is well known that Besov spaces can be
described by means of differences using the modulus of continuity of
functions and that they can be also defined, for instance in terms
of convolutions with different kinds of smooth functions.

Inspired by the work of L\"{o}fstr\"{o}m and Peetre (see \cite{L.P})
where they described for generalized tranlations, the space of
functions for which the remainder in the generalized Taylor's
formula has a given order, we define in this paper the Besov-type
space of functions associated with the Dunkl operator on the real
line, that we call Besov-Dunkl spaces of order $k$ for $k=1,2,...,$.
Before, we need to recall some results in harmonic analysis related to the
Dunkl theory.\\

For a real parameter $\alpha > -\frac{1}{2}$, the Dunkl operator on
$\mathbb{R}$ denoted by $\Lambda_{\alpha}$, is a
differential-difference operator introduced in 1989 by C. Dunkl in
\cite{dun}. This operator is associated with the reflection group $
\mathbb{Z}_{2}$ on $\mathbb{R}$ and can be considered as a
perturbation of the usual derivative by reflection part. The
operator $\Lambda_{\alpha}$ plays a major role in the study of
quantum harmonic oscillators governed by Wigner's commutation rules
(see \cite{rose}). The Dunkl kernel $E_{\alpha}$ related to
$\Lambda_{\alpha}$ is used to define the Dunkl transform
$\mathcal{F}_{\alpha}$ which enjoys properties similar to those of
the classical Fourier transform. The Dunkl kernel $E_{\alpha}$
 satisfies a product formula (see \cite{ro1}). This allows us to define
the Dunkl translation $\tau_{x}$, $x\in\mathbb{R}$. As a result, we
have the Dunkl convolution $\ast_\alpha$ (see next section).

The classical Taylor formula with integral remainder was extended to
the one dimensional Dunkl operator $\Lambda_{\alpha}$ (see
\cite{mou}): for $k=1,2,...,$ $f \in \mathcal{E}(\mathbb{R})$ and $
a \in \mathbb{R}$, we have
\begin{eqnarray*} \tau _x(f)(a) = \sum_{p=0}^{k-1} b_p(x) \Lambda_\alpha^p f(a) +  R_k(x,f)(a),\quad x \in \mathbb{\mathbb{R}}\backslash\{0\} ,\end{eqnarray*}
with $R_k(x,f)(a)$ is the integral remainder of order $k$ given by
 \begin{eqnarray*} \displaystyle R_k(x,f)(a)= \int_{-|x|}^{|x|} \Theta_{k-1} (x,y) \tau_y (\Lambda_\alpha^{k} f)(a) A_\alpha(y) dy,\end{eqnarray*}
 where $\mathcal{E}(\mathbb{R})$ is the space of infinitely
differentiable functions on $\mathbb{R}$ and $ (\Theta_{p})_{p\in
\mathbb{N}}$, $(b_p)_{p\in \mathbb{N}}$ are two sequences of
 functions constructed inductively from the function $A_\alpha$ defined on $\mathbb{R}$ by $A_\alpha(x)=
 |x|^{2\alpha+1}$ (see next section).\\

Now, we introduce the following weighted function spaces: Let $k=1, 2,...,$ $0<\beta <1$, $1 \leq p < +\infty $ and $ 1 \leq q \leq +\infty.$
\\$\bullet$ We denote by
$L^p(\mu_\alpha)$ the space of complex-valued functions $f$,
measurable on $\mathbb{R}$ such that
$$\|f\|_{p,\alpha} = \left(\int_{ \mathbb{R}}|f(x)|^p
d\mu_\alpha(x) \right)^{1/p} < + \infty,$$ where $\mu_\alpha$ is a
weighted Lebesgue measure associated with the Dunkl operator (see
next section).
\\$\bullet$ (Besov-Dunkl spaces of order $k$)
$ \mathcal{B}^k\mathcal{D}_{p,q}^{\beta,\alpha}$ denote the subspace
of functions $f\in\mathcal{E}(\mathbb{R})$ such
that $\Lambda_\alpha^{k-1}f$ are in $L^p(\mu_\alpha)$ and satisfying
\begin{eqnarray*} \int_0^{+\infty} \Big(\frac{\omega_{p,\alpha}^k(x,f)}{x^{\beta+k-1}}\Big)^q
 \frac{dx}{x} < +\infty  &if& \ \  q < +\infty \\
\mbox{and} \qquad
    \sup_{x > 0} \frac{\omega_{p,\alpha}^k(x,f)}{x^{\beta+k-1}} < +\infty   &if& \ \  q =
  +\infty,\end{eqnarray*}
with $ \omega_{p,\alpha}^k(x,f) =\displaystyle \big\| R_{k-1}(x,
f)+R_{k-1}(-x,
f)-\big(b_{k-1}(x)+b_{k-1}(-x)\big)\Lambda_\alpha^{k-1}
f\big\|_{p,\alpha}.$ Here we put for $k=1$, $\Lambda_\alpha^{0}f=f$,
$R_{0}(x, f)=\tau _x(f)$ and $R_{0}(-x, f)=\tau
_{-x}(f).$\\
$\bullet$ We put
 $$\mathcal{A}_{k}=\big\lbrace \phi \in \mathcal{S}_\ast(\mathbb{R}) : \displaystyle\int_0^{+\infty}x^{2i}\phi(x)d\mu_\alpha(x)=0,
 \, \forall i\in \{0,1,...,[\frac{k-1}{2}]\}\big\rbrace,$$
  where $ \mathcal{S}_\ast(\mathbb{R})$ is the space of even Schwartz functions
 on $\mathbb{R}$ and $[\frac{k-1}{2}]$ is the integer part of the number $\frac{k-1}{2}$. Let $\phi \in \mathcal{A}_{k}$ (see Example 4.2, section 4), we shall
 denote by $ \mathcal{C}_{p,q,\phi}^{k,\beta,\alpha}$ the subspace of functions $f$ in $\mathcal{E}(\mathbb{R})$
such that $\Lambda_\alpha^{2i}f\in L^p(\mu_\alpha),$ $0\leq i\leq
[\frac{k-1}{2}]$ and satisfying
\begin{eqnarray*} \int^{+ \infty}_0 \left(\frac{\|f \ast_\alpha
\phi_t \|_{p,\alpha}}{t^{\beta+k-1}}\right)^q \, \frac{dt}{t} < +
\infty  &if& q < +\infty
\\
\mbox{and}\qquad \sup_{t>0}\frac{\|f \ast_\alpha
\phi_t\|_{p,\alpha}}{t^{\beta+k-1}} <+\infty  & if & q =
+\infty,\end{eqnarray*} where $\phi_t$ is the dilation of $\phi$
given by $\phi_t(x)=\frac{1}{t^{2(\alpha+1)}}\phi(\frac{x}{t})$, for
all $t\in (0,+\infty)$ and $x\in\mathbb{R}$.\\

Our aim in this paper is to  generalize to the order $k=1,2,...,$ the results
obtained in \cite{ab1,An} for the case $k=1$. For this purpose, we give some properties and estimates of
the integral remainder of order $k$ and we establish that
$$ \mathcal{B}^k\mathcal{D}_{p,q}^{\beta,\alpha}\cap L^p(\mu_\alpha) =
\mathcal{C}_{p,q,\phi}^{k,\beta,\alpha}.$$
It's clear from this equality that $\mathcal{C}_{p,q,\phi}^{k,\beta,\alpha}$ is independant
 of the specific selection of the fuction $\phi$  in $\mathcal{A}_{k}$.\\

The contents of this paper are as follows. \\In section 2, we
collect some basic definitions and results about harmonic analysis
associated with the Dunkl operator $\Lambda_\alpha$. \\
In section 3, we prove some properties and estimates of the integral
remainder of order $k$. Finally, we establish the coincidence
between the characterizations of the Besov-type spaces of order $k$.\\

Along this paper, we use $c$ to represent a suitable positive
constant which is not necessarily the same in each occurrence.
\section{Preliminaries}
\label{sec:1} In this section, we recall some notations and results
in Dunkl theory on $\mathbb{R}$ and we refer for more details to \cite{am,dun,ro1}.\\

The Dunkl operator is given for $x\in \mathbb{R}$ by
$$ \Lambda_\alpha f(x) = \frac{df}{dx} (x) + \frac{2\alpha+1}{x}
 \Big[\frac{f(x)-f(-x)}{2}\Big],\; f \in \mathcal{C}^{1}(\mathbb{R}).$$
For $\lambda \in \mathbb{C}$, the initial problem
$$\Lambda_\alpha(f)(x) = \lambda f(x),\quad f(0) = 1,\quad x \in \mathbb{R},$$
has a unique solution $E_\alpha(\lambda .)$ called Dunkl kernel
given by
$$E_\alpha(\lambda x) = j_\alpha(i\lambda x) + \frac{\lambda x}
{2(\alpha+1)} j_{\alpha+1} (i\lambda x),\quad x \in \mathbb{R},$$
where $j_\alpha$ is the normalized Bessel function of the first kind
and order $\alpha.$  \\ Let $A_\alpha$ the function defined on
$\mathbb{R}$ by
$$A_\alpha(x) = |x|^{2\alpha+1},\quad x \in \mathbb{R},$$ and $\mu_\alpha$ the weighted
Lebesgue measure on $\mathbb{R}$ given by
\begin{eqnarray}d\mu_\alpha(x) = \frac{A_\alpha(x)}{2^{\alpha +1}\Gamma(\alpha
+1)}dx.\end{eqnarray} There exists an analogue of the classical
Fourier transform with respect to the Dunkl kernel called the Dunkl
transform and denoted by $\mathcal{F}_\alpha$. The Dunkl transform
enjoys properties similar to those of the classical Fourier
transform and is defined for $f \in L^1(\mu_\alpha)$ by
\begin{eqnarray*}
\mathcal{F}_\alpha(f)(x) = \int_{\mathbb{R}}f(y)\,E_\alpha(-ixy)\,
d\mu_\alpha(y), \quad x \in \mathbb{R}.
\end{eqnarray*}
For all $x, y, z \in \mathbb{R}$, we consider
 $$W_\alpha(x,y,z) = \frac{(\Gamma(\alpha+1)^2)}{2^{\alpha-1}
 \sqrt{\pi}\Gamma(\alpha + \frac{1}{2})}
  (1 - b_{x,y,z} + b_{z,x,y} +
 b_{z,y,x}) \Delta_\alpha(x,y,z)$$
where
$$b_{x,y,z} = \left\{ \begin{array}{ll}
\frac{x^2 + y^2 -z^2}{2xy} &\mbox{ if } x, y \in \mathbb{R}
\backslash \{0\},\; z \in \mathbb{R}\\
0 &\mbox{ otherwise }
\end{array}\right.$$
and
$$\Delta_\alpha(x,y,z) = \left\{ \begin{array}{ll}
\frac{([(|x| + |y|)^2 - z^2][z^2-(|x| - |y|)^2])^{\alpha -
\frac{1}{2}}}{|xyz|^{2\alpha}} &\mbox{ if } |z|\in S_{x,y}\\
0 &\mbox{ otherwise }
\end{array}\right.$$
where $$S_{x,y} = \Big[||x| - |y||\;,\; |x| + |y|\Big].$$ The kernel
$W_\alpha $, is even and we have
$$W_\alpha(x,,y,z) = W_\alpha(y,x,z) = W_\alpha(-x,z,y) =
W_\alpha(-z, y, -x)$$ and
$$\int_{\mathbb{R}}|W_\alpha(x,y,z)|d\mu_\alpha(z) \leq
\sqrt{2}.$$
 The Dunkl kernel $E_\alpha$ satisfies the following
product formula
$$E_\alpha(ixt) E_\alpha(iyt) = \int_{\mathbb{R}} E_\alpha(itz)
d\gamma_{x,y}(z),\quad x, y, t \in \mathbb{R}, $$ where
$\gamma_{x,y}$ is a signed measure on $\mathbb{R}$ given by
\begin{eqnarray}d\gamma_{x,y}(z) = \left\{
\begin{array}{ll} W_\alpha(x,y,z)d\mu_\alpha(z) &\mbox{ if } x, y
\in \mathbb{R} \backslash \{0\}\\ d\delta_x(z) &\mbox{ if } y = 0\\
d\delta_y(z) &\mbox{ if } x = 0.
\end{array}\right.\end{eqnarray} with
$\mbox{supp}\gamma_{x,y} = S_{x,y} \cup (-S_{x,y}).$\\ For $x, y \in
\mathbb{R}$ and $f$ a continuous function on $\mathbb{R}$, the Dunkl
translation operator $\tau_x$ is given by
 $$\tau_x(f)(y) =\int_{\mathbb{R}} f(z) d\gamma_{x,y}(z)$$ and satisfies the
following properties :
\begin{itemize}
\item $\tau_x$ is a continuous linear operator from
$\mathcal{E}( \mathbb{R})$ into itself.
\item For all $f \in
\mathcal{E}(\mathbb{R})$,  we have \begin{eqnarray}\tau_x(f)(y) =
\tau_y(f)(x)\quad\mbox{and}\quad \tau_0(f)(x) = f(x)\end{eqnarray}
\begin{eqnarray}\tau_x \,o\, \tau_y = \tau_y\,o\,\tau_x\quad\mbox{and}\quad
\Lambda_\alpha \,o\,\tau_x = \tau_x \,o\,\Lambda_\alpha
.\end{eqnarray}
\item For all $x \in \mathbb{R}$, the
operator $\tau_x$ extends to $L^p(\mu_\alpha),\; p \geq 1$ and we
have for $f \in L^p(\mu_\alpha)$
\begin{eqnarray}
\|\tau_x(f)\|_{p,\alpha} \leq \sqrt{2} \|f\|_{p,\alpha}.
\end{eqnarray}
\end{itemize}
The Dunkl convolution  $f\, \ast_\alpha g$ of two continuous
functions $f$ and $g$ on $\mathbb{R}$ with compact support, is
defined by
$$(f\,\ast_\alpha\, g)(x) = \int_{\mathbb{R}} \tau_x(f)(-y) g(y)
d\mu_\alpha(y),\quad x \in \mathbb{R}. $$ The convolution
$\ast_\alpha$ is associative and commutative and satisfies the
following properties:
\begin{itemize}
\item Assume that $p,q, r \in [1, + \infty[$ satisfying
$\frac{1}{p} + \frac{1}{q} = 1 + \frac{1}{r}$ (the Young condition).
Then the map $(f, g) \rightarrow f\,\ast_\alpha \,g$ defined on
$C_c(\mathbb{R}) \times C_c( \mathbb{R})$, extends to a continuous
map from $L^p(\mu_\alpha) \times L^q(\mu_\alpha)$ to
$L^r(\mu_\alpha)$ and we have \begin{eqnarray}\|f\,\ast_\alpha
\,g\|_{r,\alpha} \leq \sqrt{2}\|f\|_{p,\alpha}
\|g\|_{q,\alpha}.\end{eqnarray}
\item For all $f
\in L^1(\mu_\alpha)$, $g \in L^2(\mu_\alpha)$ and $h \in
L^p(\mu_\alpha)$, $1 \leq p < +\infty,$ we have
\begin{eqnarray*}  \mathcal{F}_\alpha(f\,\ast_\alpha g) =
\mathcal{F}_\alpha(f) \mathcal{F}_\alpha(g),\end{eqnarray*}
\begin{eqnarray}\mbox{and}\quad
 \tau_t (f\,\ast_\alpha\,h) = \tau_t(f)\,\ast_\alpha h =
f\,\ast_\alpha \tau_t(h),\; t \in \mathbb{R}.\end{eqnarray}
\end{itemize}
It has been shown in \cite{mou}, the following generalized Taylor
formula with integral remainder:
\begin{proposition} For $k=1,2,...,$ $f \in
\mathcal{E}(\mathbb{R})$ and $ a \in \mathbb{R}$, we have
\begin{eqnarray} \tau _x f(a) = \sum_{p=0}^{k-1} b_p(x) \Lambda_\alpha^p f(a) +  R_k(x,f)(a),\quad x \in \mathbb{R}\backslash\{0\} ,\end{eqnarray}
with $ R_k(x,f)(a)$ is the integral remainder of order $k$ given by
 \begin{eqnarray} \displaystyle R_k(x,f)(a)= \int_{-|x|}^{|x|} \Theta_{k-1} (x,y) \tau_y (\Lambda_\alpha^{k} f)(a) A_\alpha(y) dy,\end{eqnarray}
where\begin{itemize}
\item[i)] $\displaystyle b_{2m}(x)= \frac{1}{(\alpha+1)_m m!}  \Big(\frac{x}{2} \Big)^{2m}\;$, $\;\displaystyle
  b_{2m+1}(x)= \frac{1}{(\alpha+1)_{m+1} m!}  \Big(\frac{x}{2}
  \Big)^{2m+1}$, for all $\;m\in \mathbb{N}.$
 \item[ii)] $ \Theta_{k-1}(x,y) = u_{k-1}(x,y) + v_{k-1}(x,y)\;$
 with
   $\;\displaystyle u_0(x,y)= \frac{sgn(x)}{2 A_\alpha(x)}\, , $\\$\,\displaystyle v_0(x,y)= \frac{sgn(y)}{2
 A_\alpha(y)}$ ,
 $\quad\displaystyle u_k(x,y)= \int_{|y|}^{|x|} v_{k-1}(x,z) dz\,$ and\\ $\;\displaystyle
 v_k(x,y)= \frac{sgn(y)}{ A_\alpha(y)}\int_{|y|}^{|x|} u_{k-1}(x,z)A_\alpha(z)
   dz.$
   \end{itemize}
   \end{proposition}
   According to (\cite{ro3}, Lemma 2.2), the Dunkl operator
   $\Lambda_\alpha$ have the following regularity properties:
\begin{eqnarray} \Lambda_\alpha \;\mbox{leaves}\;\, \mathcal{C}_c^\infty(\mathbb{R})\;
 \mbox{and} \;\mbox {the\, Schwartz\, space}\; \mathcal{S}(\mathbb{R})\; \mbox{invariant}. \end{eqnarray}
\section{Some properties of the integral remainder of order $k$}
\label{sec:2} In this section, we prove some properties and
estimates of the integral remainder in the generalized Taylor
formula.
\begin{remark}
 Let $k=1,2,...,$ $f \in \mathcal{E}(\mathbb{R})$ and $ x \in
 \mathbb{R}\backslash\{0\}$.
 \begin{enumerate}
\item[1/] From Proposition 2.1, we have
  \begin{eqnarray}
R_k(x, f)& =& \tau_x(f)- f-b_1(x)\Lambda_\alpha f...- b_{k-1}(x)\Lambda_\alpha^{k-1}f \nonumber  \\
  &=&R_{k-1}(x, f)-b_{k-1}(x)\Lambda_\alpha^{k-1} f,
  \end{eqnarray} where we put for $k=1$, $R_{0}(x, f) = \tau_x(f).$
  Observe that $$R_1(x, f)=R_{0}(x, f)-b_{0}(x)\Lambda_\alpha^{0} f=\tau_x(f)-
  f.$$
  \item[2/] According to (\cite{mou}, p.352) and Proposition 2.1, i), we have
 \begin{eqnarray}
 \displaystyle \int_{-|x|}^{|x|} |\Theta_{k-1} (x,y)|  A_\alpha(y) dy &\leq& b_k(|x|)+|x| b_{k-1}(|x|)\nonumber \\
  &\leq& c\, |x|^{k}.
\end{eqnarray}
\item[3/] Note that the function $y\longmapsto
\tau_y(f)-f$ is continuous on $\mathbb{R}$ (see \cite{mou.T}, Lemma
1, (ii)), which implies that the same is true for the function
$y\longmapsto R_k(y, f).$
 \end{enumerate}
\end{remark}
\begin{lemma}
 Let $k=1,2,...,$ then there exists a constant $c>0$ such that for all $f \in \mathcal{E}(\mathbb{R})$ satisfying
 $\Lambda_\alpha^{k-1}f \in L^p(\mu_\alpha) $, we have
 \begin{equation}
  \| R_{k-1}(x,f)\|_{p,\alpha} \leq c \,|x|^{k-1} \| \Lambda_\alpha^{k-1} f \|_{p,\alpha},\quad x \in \mathbb{R}
  \backslash\{0\}.
 \end{equation}
 \end{lemma}
\begin{proof}
 Let $k=1,2,...,$ $f \in \mathcal{E}(\mathbb{R})$ such that $\Lambda_\alpha^{k-1}f \in L^p(\mu_\alpha) $ and $x \in
 \mathbb{R}\backslash\{0\}$. For $k=1$, by (2.5), it's clear that $\|R_{0}(x, f)\| =\| \tau_x(f)\|_{p,\alpha}\leq c\,\|f\|_{p,\alpha}.$
Using the Minkowski's inequality for integrals, (2.5) and (2.9), we
have for $k\geq2$
 \begin{eqnarray*}
  \| R_{k-1}(x,f)\|_{p,\alpha}  &\leq& \int_{-|x|}^{|x|}| \Theta_{k-2} (x,y)| \
   \| \tau_y ( \Lambda_\alpha^{k-1} f)\|_{p,\alpha} A_\alpha(y) dy\\
   &\leq& c \;  \|   \Lambda_\alpha^{k-1} f \|_{p,\alpha} \int_{-|x|}^{|x|} |\Theta_{k-2} (x,y)| A_\alpha(y) dy.
 \end{eqnarray*}
From (3.2),  we deduce our result.
\end{proof}
\begin{remark}  Let $k=1,2,...,$ $f \in
\mathcal{E}(\mathbb{R})$ and $x \in \mathbb{R}\backslash\{0\}.$
\begin{itemize}
\item[1/]  If $\Lambda_\alpha^{k-1}f \in L^p(\mu_\alpha) ,$  then we have by (3.1), (3.3) and
Proposition 2.1, i),
 \begin{eqnarray}
 \|R_k(x,f)\|_{p,\alpha}& \leq&
\|R_{k-1}(x,f)\|_{p,\alpha}+
    \|b_{k-1}(x) \Lambda_\alpha^{k-1} f\|_{p,\alpha} \nonumber
     \\&\leq& c \,|x|^{k-1} \|\Lambda_\alpha^{k-1} f \|_{p,\alpha}.
 \end{eqnarray}
 \item[2/] We observe from Proposition 2.1 that
 \begin{eqnarray}R_{k}(x,f)+R_{k}(-x,f)&=&\tau_x(f)+\tau_{-x}(f)-\sum_{p=0}^{k-1}
  \big(b_p(x)+b_p(-x)\big)\Lambda_\alpha^p f\nonumber
  \\&=&\tau_x(f)+\tau_{-x}(f)-2
  \sum_{i=0}^{[\frac{k-1}{2}]} b_{2i}(x)\Lambda_\alpha^{2i} f.
  \end{eqnarray}
  \end{itemize}
\end{remark}

\section{Characterizations of Besov-Dunkl spaces of order $k$}
In this section, we begin with a remark, a proposition containing sufficient
conditions and an example.
\begin{remark} Let $k=1, 2,...,$ $f\in \mathcal{E}(\mathbb{R})$ such that
$\Lambda_\alpha^{k-1}f$ is in $L^p(\mu_\alpha)$ and
$x\in(0,+\infty).$
\begin{itemize}
  \item[1/] We can assert from (3.1) that \begin{eqnarray}\omega_{p,\alpha}^k(x,f) =\| R_{k}(x, f)+R_{k}(-x,
f)\|_{p,\alpha}.\end{eqnarray}
\item[2/] For $k=1$, $
\omega_{p,\alpha}^k(x,f) = \|\tau_x(f)+\tau_{-x}(f)-
2f\|_{p,\alpha},$ called the modulus of continuity of second order
of $f$. In this case, we recover with this expression the Besov-type
spaces defined in \cite{ab1,An}.
\end{itemize}
\end{remark}
\begin{proposition} Let $k=1, 2,....$, $ 1 \leq p < +\infty $, $ 1 \leq q \leq +\infty$,
$0<\beta<1$ and $f\in \mathcal{E}(\mathbb{R})$ such that
$\Lambda_\alpha^{k-1}f$, $\Lambda_\alpha^k f$ are in
$L^p(\mu_\alpha)$, then
$f\in\mathcal{B}^k\mathcal{D}_{p,q}^{\beta,\alpha}.$
\end{proposition}
\begin{proof} Let $k=1, 2,....$, $ 1 \leq p < +\infty $, $ 1 \leq q \leq +\infty$,
$0<\beta<1$ and $f\in \mathcal{E}(\mathbb{R})$ such that
$\Lambda_\alpha^{k-1}f$, $\Lambda_\alpha^k f$ are in
$L^p(\mu_\alpha).$
 By (3.3), (3.4) and (4.1), we obtain for $x \in (0,+\infty)$
\begin{eqnarray*}\omega_{p,\alpha}^k(x,f)\leq c \,x^k \lVert \Lambda_\alpha^k f
\rVert_{p,\alpha} \quad\mbox{and}\quad \omega_{p,\alpha}^k(x,f)\leq
c \,x^{k-1} \lVert \Lambda_\alpha^{k-1} f
\rVert_{p,\alpha}.\end{eqnarray*} Then we can write,
\begin{eqnarray*}
\int_0^{+\infty}
\Big(\frac{\omega_{p,\alpha}^k(x,f)}{x^{\beta+k-1}}\Big)^q
\frac{dx}{x}\leq  c \int_0^1 \Big(\frac{\| \Lambda_\alpha^k f
\|_{p,\alpha}}{x^{\beta-1}}\Big)^q \frac{dx}{x}+c \int_1^{+\infty}
\Big(\frac{\| \Lambda_\alpha^{k-1} f
\|_{p,\alpha}}{x^{\beta}}\Big)^q \frac{dx}{x},
\end{eqnarray*} giving two finite integrals. Here when $q=+\infty$, we make the usual
modification.
\end{proof}
\begin{example} From (2.10) and Proposition 4.1, we can assert that the
spaces $\mathcal{C}_c^\infty(\mathbb{R})$ and
$\mathcal{S}(\mathbb{R})$ are included in $\mathcal{B}^k\mathcal{D}_{p,q}^{\beta,\alpha}\cap L^p(\mu_\alpha)$.
\end{example}
\label{sec:3} In order to establish that
$\mathcal{B}^k\mathcal{D}_{p,q}^{\beta,\alpha}\cap L^p(\mu_\alpha)=
\mathcal{C}_{p,q,\phi}^{k,\beta,\alpha}$, we give an example of functions in the class $\mathcal{A}_{k}$ and we prove
some useful lemmas.
\begin{example} According to (\cite{ro2}, Example 3.3,(2)), the
generalized Hermite polynomials on $\mathbb{R}$, denoted by
$H_n^{\alpha+\frac{1}{2}}$, $n\in \mathbb{N}$ are orthogonal with
respect to the measure $e^{-x^2}d\mu_\alpha(x)$ and can be written
as
$$H_{2n}^{\alpha+\frac{1}{2}}(x)=(-1)^n 2^{2n}n!\,L_n^\alpha(x^2)
\quad\mbox{and}\quad H_{2n+1}^{\alpha+\frac{1}{2}}(x)=(-1)^n
2^{2n+1}n!\,xL_n^{\alpha+1}(x^2),$$ where the $L_n^\alpha$ are the
Laguerre polynomials of index $\alpha\geq -\frac{1}{2}$, given by
$$L_n^\alpha(x)=\frac{1}{n!}\, x^{-\alpha}\,e^x
\frac{d^n}{dx^n}\Big(x^{n+\alpha}e^{-x}\Big).$$ For $k=1,2,...,$ fix
any positive integer $n_0>[\frac{k-1}{2}]$ and take for example the
function defined on $\mathbb{R}$ by $\phi(x)=H_{2
n_0}^{\alpha+\frac{1}{2}}(x)\,e^{-x^2}$. Put $P_i(x)=x^{2i}$ for
$i\in \{0,1,...,[\frac{k-1}{2}]\}$, since $P_i\in$
$span_\mathbb{R}\,\{H_p^{\alpha+\frac{1}{2}},
p=0,1,...,2[\frac{k-1}{2}]\}$, then we can assert that $ \phi \in
\mathcal{S}_\ast(\mathbb{R})$ and satisfy
$\displaystyle\int_0^{+\infty}x^{2i}\,\phi(x)\,d\mu_\alpha(x)=0$, which gives that $\phi\in \mathcal{A}_{k}$.
\end{example}
 \begin{lemma}
Let $k=1,2,...,$ $\phi \in \mathcal{A}_{k}$, $1\leq p< +\infty$ and
$r>0$, then there exists a constant $c>0$ such that for all $f\in
\mathcal{E}(\mathbb{R})\cap L^p(\mu_\alpha) $ satisfying
$\Lambda_\alpha^{k-1} f \in L^p(\mu_\alpha) $ and $t>0$, we have
\begin{eqnarray} \|\phi_t \ast_\alpha f\|_{p,\alpha} \leq
c \int_0^{+\infty}\min\Big\{\Big(\frac{x}{t}\Big)^{2(\alpha+1)},
\Big(\frac{t}{x}\Big)^{r}\Big\}\,\omega_{p,\alpha}^k(x,f)\frac{dx}{x}.\end{eqnarray}
\end{lemma}
\begin{proof} Let $k=1,2,...$, $t>0$, we have for $i\in \{0,1,...,[\frac{k-1}{2}]\}$,
\begin{eqnarray}\int_0^{+\infty}x^{2i}\phi(x)d\mu_\alpha(x)=0\;
  \Longrightarrow\int_0^{+\infty}x^{2i}\phi_t(x)d\mu_\alpha(x)=0, \end{eqnarray}
  where $\phi_{t}$ is the dilatation of $\phi$.\\
  We observe that, \begin{eqnarray*}(\phi_t \ast_\alpha
  f)(y)&=&\int_{\mathbb{R}}\phi_t(x)\tau_y(f)(-x)d\mu_\alpha(x)\\&=&\int_{\mathbb{R}}\phi_t(x)\tau_y(f)(
  x)d\mu_\alpha(x),\end{eqnarray*} then using (2.3), (3.5), (4.3) and Proposition 2.1,
   we can write for $y\in\mathbb{R}$ \\$2(\phi_t \ast_\alpha
  f)(y)$\begin{eqnarray*}&=&\int_{\mathbb{R}}\phi_t(x)\Big(\tau_y(f)(x)+\tau_y(f)(-x)-2\sum_{i=0}^{[\frac{k-1}{2}]} b_{2i}(x)
  \Lambda_\alpha^{2i} f(y)
    \Big)d\mu_\alpha(x)\\&=&
  2\int_0^{+\infty}\phi_t(x)\Big(\tau_x(f)(y)+\tau _{-x}(f)(y)-2\sum_{i=0}^{[\frac{k-1}{2}]}
   b_{2i}(x)\Lambda_\alpha^{2i} f(y)\Big)d\mu_\alpha(x)\\&=&2\int_0^{+\infty}\phi_t(x)
   \big(R_{k}(x,f)(y)+R_{k}(-x,f)(y)\big)d\mu_\alpha(x).\end{eqnarray*}
By Minkowski's inequality for integrals, we obtain\begin{eqnarray}
\;\|\phi_t \ast_\alpha f\|_{p,\alpha}&\leq&
 \int_0^{+\infty}|\phi_t(x)|\;\|R_{k}(x,f)+R_{k}(-x,f)
 \|_{p,\alpha}d\mu_\alpha(x)\nonumber\\&\leq&
 c\int_0^{+\infty}\Big(\frac{x}{t}\Big)^{2(\alpha+1)}
 \Big|\phi\Big(\frac{x}{t}\Big) \Big|\;\omega_{p,\alpha}^k(x,f)\frac{dx}{x}\\&\leq&
 c\int_0^{+\infty}\Big(\frac{x}{t}\Big)^{2(\alpha+1)} \omega_{p,\alpha}^k(x,f)\frac{dx}{x}. \end{eqnarray}
 On the other hand, since $\phi \in  \mathcal{S}_\ast( \mathbb{R})$, then from (4.4)
 and for $r>0$ there exists a constant $c$ such that
 \begin{eqnarray} \|\phi_t \ast_\alpha f\|_{p,\alpha}\leq
 c\int_0^{+\infty}\Big(\frac{t}{x}\Big)^{r} \;\omega_{p,\alpha}^k(x,f)\frac{dx}{x}
  \;.\end{eqnarray}
 Using (4.5) and (4.6), we deduce our result.
\end{proof}
\begin{lemma} Let $k=1,2,...,$ $1<p< +\infty$ and $ \phi \in \mathcal{A}_{k}$, then there exists a
constant $c>0$ such that for all $f\in \mathcal{E}(\mathbb{R})$
satisfying $\Lambda_\alpha^{2i} f \in L^p(\mu_\alpha),$ $0\leq i
\leq [\frac{k-1}{2}]$ and $x>0$, we have
\begin{eqnarray}\omega_{p,\alpha}^k(x,f) \leq c \int_0^{+\infty}
\min\Big\{\Big(\frac{x}{t}\Big)^{k-1},\Big(\frac{x}{t}\Big)^{k}\Big\}\|\phi_t
\ast_\alpha f\|_{p,\alpha} \frac{dt}{t}\;.
\end{eqnarray}
\end{lemma}
\begin{proof} Put for $0<\varepsilon<\delta<+\infty$
\begin{eqnarray*}
f_{\varepsilon,\delta}(y)=\int_\varepsilon^\delta
 (\phi_t \ast_\alpha \phi_t \ast_\alpha f)(y)\frac{dt}{t}\;,\;\;\;y\in
 \mathbb{R}.\end{eqnarray*}
 Then for $i \in \mathbb{N}$, we have
\begin{eqnarray*}
\Lambda_{\alpha}^{2i}f_{\varepsilon,\delta}(y)=\int_\varepsilon^\delta
 (\Lambda_{\alpha}^{2i}\phi_t \ast_\alpha  \phi_t \ast_\alpha f)(y)\frac{dt}{t}\;,\;\;\;y\in \mathbb{R}.\end{eqnarray*}
From the integral representation of $\tau_x,$ we obtain by
interchanging the orders of integration and (2.7),
   \begin{eqnarray*}\tau_x(f_{\varepsilon,\delta})(y)&=&\int_\varepsilon^\delta
 \tau_x(\phi_t \ast_\alpha \phi_t \ast_\alpha f)(y)\frac{dt}{t}\\ &=&\int_\varepsilon^\delta
 (\tau_x (\phi_t) \ast_\alpha \phi_t \ast_\alpha f)(y)\frac{dt}{t} \;,\;\;y\in \mathbb{R},\;x\in(0,+\infty),\end{eqnarray*} so we
 can write for $x\in(0,+\infty)$ and $y\in \mathbb{R}$,\\
  $(R_{k}(x,f_{\varepsilon,\delta})+R_{k}(-x,f_{\varepsilon,\delta}))(y)$\begin{eqnarray*} =
 \displaystyle \int_\varepsilon^\delta
 \big[\big(\tau_x(\phi_t)+\tau_{-x}(\phi_t)-2 \sum_{i=0}^{[\frac{k-1}{2}]} b_{2i}(x)
 \Lambda_{\alpha}^{2i}\phi_t\big) \ast_\alpha \phi_t \ast_\alpha f\big](y)\frac{dt}{t}\,.\end{eqnarray*}
 Using Minkowski's inequality for integrals and (2.6), we get\\
$\|(R_{k}(x,f_{\varepsilon,\delta})+R_{k}(-x,f_{\varepsilon,\delta}))\|_{p,\alpha}$
 \begin{eqnarray}
 &\leq&\int_\varepsilon^\delta
 \|(\tau_x(\phi_t)+\tau_{-x}(\phi_t)-2 \displaystyle\sum_{i=0}^{[\frac{k-1}{2}]} b_{2i}(x)\Lambda_{\alpha}^{2i}\phi_t)
  \ast_\alpha \phi_t \ast_\alpha f\|_{p,\alpha}\frac{dt}{t}\nonumber \\&\leq& c\int_\varepsilon^\delta
 \| (\tau_x(\phi_t)+\tau_{-x}(\phi_t)-2 \displaystyle\sum_{i=0}^{[\frac{k-1}{2}]} b_{2i}(x)\Lambda_{\alpha}^{2i}\phi_t) \|_{1,\alpha}
 \|\phi_t \ast_\alpha f\|_{p,\alpha}\frac{dt}{t}\nonumber\\
 &=& c\int_\varepsilon^\delta
\|R_{k}(x,\phi_{t})+R_{k}(-x,\phi_{t})\|_{1,\alpha}
 \|\phi_t \ast_\alpha f\|_{p,\alpha}\frac{dt}{t}\;.\end{eqnarray}
 For $x,\;t\in (0,+\infty)$, we have \\$ \|R_{k}(x,\phi_{t})+R_{k}(-x,\phi_{t})\|_{1,\alpha}$
 \begin{eqnarray*} &=&\| \tau_x(\phi_t)+\tau_{-x}(\phi_t)-2
 \sum_{i=0}^{[\frac{k-1}{2}]} b_{2i}(x)\Lambda_{\alpha}^{2i}\phi_t
 \|_{1,\alpha}\nonumber\\
 &=&\int_{\mathbb{R}}\Big|\Big(\int_{\mathbb{R}}\phi_t(z)\big(d\gamma_{x,y}(z)+d\gamma_{-x,y}(z)\big)\Big)-2 \sum_{i=0}^{[\frac{k-1}{2}]}
 b_{2i}(x)\Lambda_{\alpha}^{2i}\phi_t(y)\Big|d\mu_\alpha(y)\nonumber\\
 &=&\int_{\mathbb{R}}\Big|\Big(\int_{\mathbb{R}}\phi\big( \frac{z}{t}\big)\big(d\gamma_{x,y}(z)+d\gamma_{-x,y}(z)\big)\Big)-2\sum_{i=0}^{[\frac{k-1}{2}]} b_{2i}\big(\frac{x}{t}\big)\Lambda_{\alpha}^{2i}\phi\big(
 \frac{y}{t}\big)\Big|\frac{1}{t^{2(\alpha+1)}}d\mu_\alpha(y)\;.\end{eqnarray*}
 By (2.2) and the change of variable $z'=\frac{z}{t}$, we have
 $$ W_\alpha(x,y,z't)\;t^{2(\alpha+1)}=W_\alpha( \frac{x}{t}, \frac{y}{t},z'),$$
 which implies that
  $\displaystyle d\gamma_{x,y}(z)=d\gamma_{\frac{x}{t},\frac{y}{t}}(z')\,.$
  Hence, we obtain\\  $ \|R_{k}(x,\phi_{t})+R_{k}(-x,\phi_{t})\|_{1,\alpha}$
   \begin{eqnarray*}
   &=&\int_{\mathbb{R}}\Big|\Big(\int_{\mathbb{R}}\phi( z')\big(d\gamma_{\frac{x}{t},\frac{y}{t}}(z')
  +d\gamma_{\frac{-x}{t},\frac{y}{t}}(z')\big)\Big)-2\sum_{i=0}^{[\frac{k-1}{2}]} b_{2i}\big(\frac{x}{t}\big)
  \Lambda_{\alpha}^{2i}\phi \big(\frac{y}{t}\big)\Big|\frac{1}{t^{2(\alpha+1)}}d\mu_\alpha(y)
 \\ &=&\int_\mathbb{R}\Big|\Big(\tau_\frac{x}{t}(\phi)\big(\frac{y}{t}\big)+\tau_\frac{-x}{t}(\phi)\big(\frac{y}{t}\big)
 \Big)
 \frac{1}{t^{2(\alpha+1)}}
  -2\Big(\sum_{i=0}^{[\frac{k-1}{2}]} b_{2i}\big(\frac{x}{t}\big)\Lambda_{\alpha}^{2i}\phi \Big)_t(y) \Big|d\mu_\alpha(y)
 \\
  &=& \Big\|\Big(\tau_\frac{x}{t}(\phi) +\tau_\frac{-x}{t}(\phi)
     -2\sum_{i=0}^{[\frac{k-1}{2}]}
     b_{2i}\big(\frac{x}{t}\big)\Lambda_{\alpha}^{2i}\phi \Big)_{t}
   \Big\|_{1,\alpha} \\&=& \Big\| \tau_\frac{x}{t}(\phi) +\tau_\frac{-x}{t}(\phi)
     -2 \sum_{i=0}^{[\frac{k-1}{2}]} b_{2i}\big(\frac{x}{t}\big)\Lambda_{\alpha}^{2i}\phi
   \Big\|_{1,\alpha},\end{eqnarray*} which gives
   \begin{eqnarray}\|R_{k}(x,\phi_{t})+ R_{k}(-x,\phi_{t})\|_{1,\alpha}= \Big\|
R_{k}(\frac{x}{t},\phi)+R_{k}(\frac{-x}{t},\phi)
 \Big\|_{1,\alpha}. \end{eqnarray}
  Since $\phi \in
\mathcal{S}_\ast(\mathbb{R})$, then using (2.10) and (3.3), we can
assert that
$$\Big\| R_{k}(\frac{x}{t},\phi)+R_{k}(\frac{-x}{t},\phi)
 \Big\|_{1,\alpha} \leq
c\;\big(\frac{x}{t}\big)^{k}\|\Lambda_{\alpha}^{k}\phi\|_{1,\alpha}\leq
c\;\big(\frac{x}{t}\big)^{k},$$
    on the other hand, by (3.4) we have $$\Big\|
R_{k}(\frac{x}{t},\phi)+R_{k}(\frac{-x}{t},\phi)
 \Big\|_{1,\alpha} \leq c\;\big(\frac{x}{t}\big)^{k-1}\|\Lambda_{\alpha}^{k-1}\phi\|_{1,\alpha}\leq
c\;\big(\frac{x}{t}\big)^{k-1}, $$
     then we get,\begin{eqnarray}\Big\| R_{k}(\frac{x}{t},\phi)+R_{k}(\frac{-x}{t},\phi) \Big\|_{1,\alpha}\leq c\;\min \Big\{\big(\frac{x}{t}\big)^{k-1},\big(\frac{x}{t}\big)^{k}\Big\}.\qquad\quad \;\end{eqnarray}
     From (3.6), (4.8), (4.9) and (4.10), we obtain
 \begin{eqnarray}
 \omega_{p,\alpha}^k(x,f_{\varepsilon,\delta})
  \leq c \int_\varepsilon^{\delta}
\min\Big\{\big(\frac{x}{t}\big)^{k-1},\big(\frac{x}{t}\big)^{k}\Big\}\|\phi_t
\ast_\alpha f\|_{p,\alpha} \frac{dt}{t}\;.\end{eqnarray}
 Note that $\Lambda_{\alpha}^{2i}\phi\ast_\alpha\phi\in\mathcal{S_*}( \mathbb{R}).$ By
 (2.1) and
 (2.7), we have
\begin{eqnarray*}\int_\mathbb{R}(\Lambda_{\alpha}^{2i}\phi\ast_\alpha
\phi)(x)|x|^{2\alpha+1}dx&=&2^{\alpha+1}
\Gamma(\alpha+1)\mathcal{F}_\alpha(\Lambda_{\alpha}^{2i}\phi\ast_\alpha
\phi)(0)\\&=&2^{\alpha+1}
\Gamma(\alpha+1)\mathcal{F}_\alpha(\Lambda_{\alpha}^{2i}\phi)(0)\mathcal{F}_\alpha(\phi)(0)\\&=&2^{\alpha+1}
\Gamma(\alpha+1) \mathcal{F}_\alpha(\Lambda_{\alpha}^{2i}\phi)(0)
\int_\mathbb{R} \phi(z)d\mu_\alpha(z)=0.
\end{eqnarray*} Since $\Lambda_{\alpha}^{2i}\phi\ast_\alpha\phi$ is in the Schwartz space $\mathcal{S}( \mathbb{R})$, we have $$\int_\mathbb{R}|log|x||\;|\Lambda_{\alpha}^{2i}\phi\ast_\alpha\phi(x)|\;|x|^{2\alpha+1}dx<+\infty.$$
 Then, by Calder\'on's reproducing formula related to the Dunkl operator (see \cite{mou.T}, Theorem 3), we have
 $$\lim_{\varepsilon\rightarrow0,\;\delta\rightarrow +\infty}
 \Lambda_{\alpha}^{2i}f_{\varepsilon,\delta} = c \;\Lambda_{\alpha}^{2i}f\;,\;\;\;in\;L^p(\mu_\alpha)\;,$$
 hence from (4.11), we deduce our result.
 \end{proof}
\begin{theorem} Let $0 <\beta <1$, $ k=1,2,...,$ $ 1 < p < +\infty $ and $ 1 \leq q \leq +\infty$, then we have
 \begin{eqnarray*}
 \mathcal{B}^k\mathcal{D}_{p,q}^{\beta,\alpha}\cap L^p(\mu_\alpha)=
   \mathcal{C}_{p,q,\phi}^{k,\beta,\alpha},\end{eqnarray*} and for $p=1$,
   we have only $\mathcal{B}^k\mathcal{D}_{1,q}^{\beta,\alpha}\cap L^p(\mu_\alpha)\subset
   \mathcal{C}_{1,q,\phi}^{k,\beta,\alpha}.$
\end{theorem}
\begin{proof} Assume $f\in
\mathcal{B}^k\mathcal{D}_{p,q}^{\beta,\alpha}\cap L^p(\mu_\alpha)$ for
$ 1 \leq p < +\infty $, $ 1 \leq q \leq +\infty$ and $r>\beta+k-1$.\\

$\bullet$ Case $q=1$. By (4.2) and Fubini's theorem, we have
\\$\displaystyle{\int_0^{+\infty}\frac{\|f \ast_\alpha \phi_t
\|_{p,\alpha}}{t^{\beta+k-1}}\frac{dt}{t}}$
\begin{eqnarray*}
  &\leq& c \int_0^{+\infty}\int_0^{+\infty}\min\Big\{\Big(\frac{x}{t}\Big)^{2(\alpha+1)},
\Big(\frac{t}{x}\Big)^{r}\Big\}\omega_{p,\alpha}^k(x,f)t^{-\beta-k}dt\frac{dx}{x}\\
 &\leq& c \int_0^{+\infty}\omega_{p,\alpha}^k(x,f)\Big(\int_0^{+\infty}
 \min\Big\{\Big(\frac{x}{t}\Big)^{2(\alpha+1)},
\Big(\frac{t}{x}\Big)^{r}\Big\}t^{-\beta-k}dt\Big)\frac{dx}{x}\\
 &\leq& c \int_0^{+\infty}\omega_{p,\alpha}^k(x,f)\Big(x^{-r}\int_0^xt^{r-\beta-k}dt+x^{2(\alpha+1)}
 \int_x^{+\infty}t^{-\beta-k-2\alpha-2}dt\Big)\frac{dx}{x}\\
  &\leq&c \int_0^{+\infty}\frac{\omega_{p,\alpha}^k(x,f)}{x^{\beta+k-1}}\frac{dx}{x}<+\infty,\end{eqnarray*}hence
 $f\in \mathcal{C}_{p,1,\phi}^{k,\beta,\alpha}$.

$\bullet$ Case $q=+\infty$. By (4.2), we have\\$\|\phi_t \ast_\alpha
f\|_{p,\alpha}$
\begin{eqnarray*}  & \leq & c\;
\Big(\int_0^{t}\Big(\frac{x}{t}\Big)^{2(\alpha+1)}\omega_{p,\alpha}^k(x,f)\frac{dx}{x}
+\int_t^{+\infty}\Big(\frac{t}{x}\Big)^{r}\omega_{p,\alpha}^k(x,f)\frac{dx}{x}\Big)
\\&\leq& c
\sup_{x\in(0,+\infty)}\frac{\omega_{p,\alpha}^k(x,f)}{x^{\beta+k-1}}
\Big(t^{-2(\alpha+1)}\int_0^tx^{2\alpha+\beta+k}dx
+t^r\int_t^{+\infty}x^{\beta+k-r-2}dx\Big)\\&\leq&
c\;t^{\beta+k-1}\sup_{x\in(0,+\infty)}\frac{\omega_{p,\alpha}^k(x,f)}{x^{\beta+k-1}},\end{eqnarray*}
then we deduce that $f\in\mathcal{C}_{p,\infty,\phi}^{k,\beta,\alpha}$.

$\bullet$ Case $1<q<+\infty$. By (4.2) again, we have for $t>0$
$$\frac{\|\phi_t \ast_\alpha f\|_{p,\alpha}}{t^{\beta+k-1}}  \leq c
\int_0^{+\infty}\Big(\frac{x}{t}\Big)^{\beta+k-1}
\min\Big\{\Big(\frac{x}{t}\Big)^{2(\alpha+1)},
\Big(\frac{t}{x}\Big)^{r}\Big\}\frac{\omega_{p,\alpha}^k(x,f)}{x^{\beta+k-1}}\frac{dx}{x}
\;.$$ Put $\displaystyle{L(x,t)=\Big(\frac{x}{t}\Big)^{\beta+k-1}
\min\Big\{\Big(\frac{x}{t}\Big)^{2(\alpha+1)},
\Big(\frac{t}{x}\Big)^{r}\Big\}}$ and
 $\displaystyle{q'=\frac{q}{q-1}}$ the conjugate of $q$. Since $$\displaystyle{\int_0^{+\infty} L(x,t) \frac{dx}{x}}
 =t^{-\beta-k-2\alpha-1}\int_0^t
x^{\beta+k+2\alpha}dx+
 t^{-\beta-k+r+1}\int_t^{+\infty} x^{\beta+k-r-2}dx\leq c,$$ we can write using H\"older's inequality,
\begin{eqnarray*} \frac{\|\phi_t \ast_\alpha
f\|_{p,\alpha}}{t^{\beta+k-1}} &\leq&c
\int_0^{+\infty}(L(x,t))^{\frac{1}{q'}}\Big((L(x,t))^{\frac{1}{q}}\frac{\omega_{p,\alpha}^k(x,f)}{x^{\beta+k-1}}\Big)\frac{dx}{x}
\\&\leq&
c\;\Big(\int_0^{+\infty}L(x,t)\Big(\frac{\omega_{p,\alpha}^k(x,f)}{x^{\beta+k-1}}\Big)^q\frac{dx}{x}\Big)^{\frac{1}{q}}.\end{eqnarray*}
By the fact that $$ \int_0^{+\infty}L(x,t)\frac{dt}{t} =
x^{\beta+k-r-1}\int_0^x t^{-\beta-k+r}dt+
 x^{\beta+k+2\alpha+1}\int_x^{+\infty} t^{-\beta-k-2\alpha-2}dt\leq c,$$ we get using Fubini's theorem,
\begin{eqnarray*} \int_0^{+\infty}\Big(\frac{\|\phi_t \ast_\alpha
f\|_{p,\alpha}}{t^{\beta+k-1}}\Big)^q \frac{dt}{t}
&\leq& c \int_0^{+\infty}\Big(\frac{\omega_{p,\alpha}^k(x,f)}{x^{\beta+k-1}}\Big)^q\Big(\int_0^{+\infty}L(x,t)\frac{dt}{t}\Big)\frac{dx}{x}\\
&\leq& c
\int_0^{+\infty}\Big(\frac{\omega_{p,\alpha}^k(x,f)}{x^{\beta+k-1}}\Big)^q
\frac{dx}{x}< +\infty ,
\end{eqnarray*} which proves the result.\\
Assume now $f\in\mathcal{C}_{p,q,\phi}^{k,\beta,\alpha}$ for $1<
p<+\infty$ and $1\leq q\leq+\infty.$

$\bullet$ Case $q=1$. By (4.7) and Fubini's theorem, we
have\\$\displaystyle\int_0^{+\infty}\frac{\omega_p^\alpha(f)(x)}{x^{\beta+k-1}}\frac{dx}{x}$
\begin{eqnarray*}
 &\leq & c \int_0^{+\infty}\int_0^{+\infty}
\min\Big\{\big(\frac{x}{t}\big)^{k-1},\big(\frac{x}{t}\big)^{k}\Big\}\|\phi_t
\ast_\alpha f\|_{p,\alpha}x^{-\beta-k} \frac{dt}{t}dx\\&\leq&
c\int_0^{+\infty}\|\phi_t \ast_\alpha
f\|_{p,\alpha}\Big(\int_0^{+\infty}
\min\Big\{\big(\frac{x}{t}\big)^{k-1},\big(\frac{x}{t}\big)^{k}\Big\}x^{-\beta-k}dx\Big)
\frac{dt}{t}\\&\leq& c\int_0^{+\infty}\|\phi_t \ast_\alpha
f\|_{p,\alpha}\Big(\frac{1}{t^k}\int_0^t
x^{-\beta}dx+\frac{1}{t^{k-1}}\int_t^{+\infty}
x^{-\beta-1}dx\Big)\frac{dt}{t}\\&\leq& c
\int_0^{+\infty}\frac{\|\phi_t \ast_\alpha
f\|_{p,\alpha}}{t^{\beta+k-1}}\frac{dt}{t}<+\infty,\end{eqnarray*}
then we obtain the result.

$\bullet$ Case $q=+\infty$. By (4.7), we get
\begin{eqnarray*} \omega_p^\alpha(f)(x) &\leq& c\;
\Big(\int_0^{x}\big(\frac{x}{t}\big)^{k-1}\|\phi_t \ast_\alpha
f\|_{p,\alpha}\frac{dt}{t}+
 \int_x^{+\infty} \big(\frac{x}{t}\big)^{k} \|\phi_t \ast_\alpha f\|_{p,\alpha}
\frac{dt}{t}\Big)\\&\leq& c \sup_{t\in(0,+\infty)}\frac{\|\phi_t
\ast_\alpha f\|_{p,\alpha}}{t^{\beta+k-1}}
\Big(x^{k-1}\int_0^{x}t^{\beta-1}dt + x^k
\int_x^{+\infty}t^{\beta-2} dt \Big)\\&
 \leq& c \;x^{\beta+k-1} \sup_{t\in(0,+\infty)}\frac{\|\phi_t \ast_\alpha f\|_{p,\alpha}}
{t^{\beta+k-1}},\end{eqnarray*} so, we deduce that
$f\in\mathcal{B}^k\mathcal{D}_{p,\infty}^{\beta,\alpha}\cap L^p(\mu_\alpha)$.

$\bullet$ Case $1<q<+\infty$. By (4.7) again, we have for $x>0$
$$  \frac{\omega_p^\alpha(f)(x)}{x^{\beta+k-1}} \leq c
\int_0^{+\infty}\Big(\frac{t}{x}\Big)^{\beta+k-1}
\min\Big\{\big(\frac{x}{t}\big)^{k-1},\big(\frac{x}{t}\big)^{k}\Big\}\frac{
\|\phi_t \ast_\alpha f\|_{p,\alpha}}{t^{\beta+k-1}}
\frac{dt}{t}\;.$$ Note that $$\Big(\frac{t}{x}\Big)^{\beta+k-1}
\min\Big\{\big(\frac{x}{t}\big)^{k-1},\big(\frac{x}{t}\big)^{k}\Big\}=\Big(\frac{t}{x}\Big)^\beta
\min\Big\{1,\frac{x}{t}\Big\}.$$ Put
$\displaystyle{G(x,t)=\Big(\frac{t}{x}\Big)^\beta
\min\Big\{1,\frac{x}{t}\Big\}}$ and
 $q'$ the conjugate of $q.$ Since
$$\displaystyle{\int_0^{+\infty} G(x,t) \frac{dt}{t}}
 =x^{-\beta}\int_0^x
t^{\beta-1}dt+
 x^{-\beta+1}\int_x^{+\infty} t^{\beta-2}dt\leq c,$$ then using H\"older's inequality, we can write
  \begin{eqnarray*}\frac{\omega_p^\alpha(f)(x)}{x^{\beta+k-1}}
  &\leq& c \int_0^{+\infty}(G(x,t))^{\frac{1}{q'}}\Big((G(x,t))^{\frac{1}{q}}
\frac{ \|\phi_t \ast_\alpha f\|_{p,\alpha}}{t^{\beta+k-1}}\Big)
\frac{dt}{t}\\&\leq& c \;\Big(\int_0^{+\infty}  G(x,t) \Big( \frac{
\|\phi_t \ast_\alpha f\|_{p,\alpha}}{t^{\beta+k-1}}\Big)^q
\frac{dt}{t}\Big)^{\frac{1}{q}}.\end{eqnarray*} By the fact that
$$\int_0^{+\infty}G(x,t)\frac{dx}{x} =
t^{\beta-1}\int_0^t x^{-\beta}dx+
 t^{\beta}\int_t^{+\infty} x^{-\beta-1}dx
 \leq c,$$  we
get using Fubini's theorem,
\begin{eqnarray*}
\int_0^{+\infty}\Big(\frac{\omega_p^\alpha(f)(x)}{x^{\beta+k-1}}\Big)^q
\frac{dx}{x}&\leq& c\int_0^{+\infty}\Big(\frac{ \|\phi_t \ast_\alpha
f\|_{p,\alpha}}{t^{\beta+k-1}}\Big)^q\Big(\int_0^{+\infty} G(x,t)
\frac{dx}{x}\Big)\frac{dt}{t}\\&\leq& c\int_0^{+\infty}\Big(\frac{
\|\phi_t \ast_\alpha f\|_{p,\alpha}}{t^{\beta+k-1}}\Big)^q
\frac{dt}{t} <+\infty,
\end{eqnarray*} thus the result is established.
\end{proof}
\begin{remark}
From theorem 4.1, we can assert that $\mathcal{C}_{p,q,\phi}^{k,\beta,\alpha}$ is independant of the specific selection of the function $\phi$ in $\mathcal{A}_{k}.$
\end{remark}

\end{document}